\documentclass[oneside,reqno]{amsart}

\usepackage{amssymb}
\usepackage{amsmath}
\usepackage{amsthm}
\usepackage{amsbsy}
\usepackage{enumerate}
\usepackage{bm}
\usepackage{hyperref}
\date{\today}
\usepackage{cite}
\usepackage{array}
\usepackage{xcolor}
\usepackage{tikz}
\usepackage{graphics}
\usepackage{subcaption}
\newtheorem{localclaim}{Claim}

\makeatletter
\newcommand{\distas}[1]{\mathbin{\overset{#1}{\kern\z@\sim}}}

\newcommand{\distras}[1]{
  \mathbin{\overset{#1}{\kern\z@\resizebox{\wd\mybox}{\ht\mysim}{$\sim$}}}
}

    \newtheorem{theorem}{Theorem}
    \newtheorem{lemma}[theorem]{Lemma}
    \newtheorem{proposition}[theorem]{Proposition}

\theoremstyle{definition} % For roman text in the body
    \newtheorem{definition}[theorem]{Definition}
    
    \newtheorem{result}[theorem]{Result}
    \newtheorem{remark}[theorem]{Remark}
    \newtheorem{example}[theorem]{Example}
    \newtheorem{exercise}[theorem]{Exercise}

\def\sig{\sigma}

\newcommand{\E}{\mbox{E}}

\def\N{\mathbb{N}}
\def\PP{{\mathcal P}}

\def\tends{\rightarrow}

\def\<{\langle}
\def\>{\rangle}

\def\P{\mbox{P}}
\newcommand\Tr{{\mbox{Tr}}}

\newcommand\mnote[1]{} %off
\newcommand\be{\begin{equation*}}

\newcommand\ee{\end{equation*}}

\newcommand\ben{\begin{equation}}
\newcommand\een{\end{equation}}
\newcommand\bes{\begin{eqnarray*}}
\newcommand\ees{\end{eqnarray*}}

\newcommand\bex{\begin{exercise}}
\newcommand\eex{\end{exercise}}
\newcommand\beg{\begin{example}}
\newcommand\eeg{\end{example}}
\newcommand\benu{\begin{enumerate}}
\newcommand\eenu{\end{enumerate}}
\newcommand\beit{\begin{itemize}}
\newcommand\eeit{\end{itemize}}
\newcommand\berk{\begin{remark}}
\newcommand\eerk{\end{remark}}
\newcommand\bdefn{\begin{defintion}}
\newcommand\edefn{\end{definition}}
\newcommand\bthm{\begin{theorem}}
\newcommand\ethm{\end{theorem}}
\newcommand\bprf{\begin{proof}}
\newcommand\eprf{\end{proof}}
\newcommand\blem{\begin{lemma}}
\newcommand\elem{\end{lemma}}

\newcommand{\var}{\mbox{\rm Var}}

\newcommand{\Cov}{\mbox{\rm Cov}}
\newcommand{\sm}{{\raise0.3ex\hbox{$\scriptstyle \setminus$}}}

\def\sig{\sigma}

\def\tends{\rightarrow}

\def\CHI{\mathchoice%
{\raise2pt\hbox{$\chi$}}%
{\raise2pt\hbox{$\chi$}}%
{\raise1.3pt\hbox{$\scriptstyle\chi$}}%
{\raise0.8pt\hbox{$\scriptscriptstyle\chi$}}}
\def\smalloplus{\raise1pt\hbox{$\,\scriptstyle \oplus\;$}}

\title [Linear eigenvalue statistics of Hankel matrices] {Asymptotic Behaviour of Linear eigenvalue statistics of Hankel matrices}

\author{Kiran Kumar A.S.}
\address{Department of Mathematics\\
        Indian Institute of Technology Bombay\\
         Powai, Mumbai, Maharashtra 400076, India}
 \email{kiran [at] math.iitb.ac.in}

\author{Shambhu Nath Maurya}
\address{Department of Mathematics\\
	Indian Institute of Technology Bombay\\
	Powai, Mumbai, Maharashtra 400076, India}

\email{snmaurya [at] math.iitb.ac.in}

\date{\today}
\thanks{The research of Kiran Kumar A.S. is supported by IIT Bombay
and the work of Shambhu Nath Maurya is partially supported by UGC Doctoral Fellowship, India.}
\begin{document}

\begin{abstract}
   We study linear eigenvalue statistics of band Hankel matrices with Brownian motion entries. We prove that, the centred, normalized linear eigenvalue statistics of band Hankel matrices obey a central limit theorem (CLT) type result. We also discuss the convergence of linear eigenvalue statistics of band Hankel matrices with independent entries for odd power monomials. Our method is based on trace formula, moment method and some results of process convergence.

\end{abstract}

\maketitle

 
\section{Introduction}

For an $n \times n$ matrix $A_n$, the linear eigenvalue statistics is defined as 
\begin{equation} \label{eqn:LES}
\mathcal{A}_n(\phi)= \sum_{i=1}^{n} \phi(\lambda_i),
\end{equation} 
where $\lambda_1 , \lambda_2 , \ldots , \lambda_n$ are the eigenvalues of $A_n$ and $\phi$ is a `nice' test function. The study of linear eigenvalue statistics is a major area of interest in random matrix theory. For an overview of the fluctuations of linear eigenvalue statistics for various types of random matrices,  an interested reader can look into the following papers, \cite{bai2004clt}, \cite{jana2014}, \cite{li_sun_2015}, \cite{adhikari_saha2018}, \cite{lytova2009central}, \cite{m&s_toeplitz_2020}, \cite{a_m&s_circulaun2020}.

Hankel and the closely related Toeplitz are two important classes of random matrices. In particular, Hankel matrices show up prominently in studies of Pad\'{e} approximation, moment problems and orthogonal polynomials (see\cite{brezinski_pade}, \cite{naum_book}). In \cite{bai1999}, Bai proposed the study of limiting spectral distributions (LSD) of Hankel and the closely related Toeplitz matrix. The existence of LSD for Hankel matrices was shown independently by Hammond and Miller  \cite{hammond&miller} and Bryc et al. \cite{bryc}. Later with certain band conditions, Basak, Bose \cite{basak} and Liu, Wang \cite{liu} independently found the limit spectral distribution of band Hankel matrices. In 2012, it was proved by Liu et al. that the linear eigenvalue statistics obey a central limit theorem type of result  for $\phi(x) = x^{2k}$ for $k \in \N$.

% the Hankel matrix is defined as 
%$$
%H_n=\left(\begin{array}{cccccc}
%x_1 & x_2 & x_3 & \ldots & x_{n-1} & x_n \\
%x_2 & x_3 & x_4 & \ldots & x_{n} & x_{n+1}\\
%x_3 & x_4 & x_5 & \ldots & x_{1} & x_{n+2} \\
%\vdots & \vdots & {\vdots} & \ddots & {\vdots} & \vdots \\
%x_n & x_{n+1} & x_{n+2} & \ldots & x_{2n-2} & x_{2n-1}
%\end{array}\right).
%$$

%$H_n=(x_{i+j-1})_{i,j=1}^n$. Hankel matrices are closely related to another important class of matrices known as Toeplitz matrix given by $T_n= (x_{i-j})_{i,j=1}^n$. More specifically, for any Toeplitz matrix $T_n$, $P_nT_n$ is a Hankel matrix and conversely for any Hankel matrix $H_n$, $P_n H_n$ is a  Toeplitz matrix. Since $P_n^{-1} = P_n$, we obtain that any Hankel matrix(and any Toeplitz matrix) is of this form.
The fluctuation problem we are interested to consider in this article, is inspired by the  results on band Toeplitz matrices by Li and Sun \cite{li_sun_2015}.
The authors  studied time dependent fluctuations of linear eigenvalue statistics for band Toeplitz matrices with standard Brownian motion entries. 
We follow the definition of Hankel matrix adopted in \cite{liu2012fluctuations}. 

Consider an input sequence $\{a_n(t);t\geq 0\}_{n\geq 1}$ of independent
standard Brownian motions along with $a_0(t)\equiv 0$. For a sequence $\{b_n\}$, with $b_n \rightarrow \infty$ and $b_n/n \rightarrow b$,  we define a band Hankel matrix $H_n(t)$ as 
\begin{equation} \label{def:A_n(t)}
H_n = P_n T_n(t)  \mbox{ and }  A_n(t) := \frac{1}{\sqrt{b_n}}  H_n(t),
\end{equation}
where $T_n= (a_{i-j}(t)\delta_{|i-j|\le b_n})$ is the band Toeplitz matrix and $P_n = (\delta_{i-1, n-j} )_{i,j=1}^n$ is the backward identity permutation.

Now for each $p \in \N$, define
\begin{equation}\label{eqn:w_p(t)}
   w_{p}(t) := \frac{\sqrt{b_n}}{n}\big(\Tr(A_n(t))^{p}-\E\Tr(A_n(t))^{p}\big). 
   \end{equation}
Observe that $\Tr(A_n(t))^{p}$ is a linear eigenvalue statistics of $A_n(t)$ as defined in \eqref{eqn:LES} with test function $\phi(x)=x^{p}$. Here, $n$ is suppressed in the notation of $w_{p}(t)$ to keep the notation simple.

The following theorem describes the process convergence of $\{ w_{p}(t) ; t \geq 0\}$ for even integer $p \geq 2$, inspired from the results in \cite{m&s_toeplitz_2020}, \cite{a_m&s_circulaun2020}.
\begin{theorem}\label{thm:hankprocess}
	Suppose $ p\geq 2$ is an even positive integer and $b_n=o(n)$. Then as $n\to\infty$
	\begin{equation*}
	\{ w_{p}(t) ; t \geq 0\} \stackrel{\mathcal D}{\rightarrow} \{W_{p}(t) ;  t \geq 0\}, 
	\end{equation*}
	where 	$\{W_{p}(t);t\geq 0, p\geq 2\}$ are Gaussian with
	 mean zero and the following covariance structure:
	\begin{equation}\label{eqn:W_p(t):cov_hp}
		\Cov[ W_p(t_1), W_{q}(t_2)]
		= 2^{\frac{p+q}{2}} \sum_{r=2,4, \ldots, q}{q \choose r}t_{1}^{\frac{p+r}{2}} \left(t_{2}-t_{1}\right)^{\frac{q-r}{2}}  R(p, r) \Gamma\left(\frac{q-r}{2}+1 \right),
	\end{equation}
	with $ R(p, r)$ as in (\ref{eqn:Cov_wp_hp}), $\Gamma(\cdot)$ denotes the Gamma function and $\stackrel{\mathcal D}{\rightarrow}$ denotes the weak convergence as in Definition \ref{processconvergence}.
\end{theorem}

In Theorem \ref{thm:hankprocess}, we have considered $a_0(t)\equiv 0$. 
A generalization of Theorem \ref{thm:hankprocess} for standard Brownian motion $\{a_0(t);t\ge 0\}$ is discussed in Section \ref{sec:process convergence} after the proof of Theorem \ref{thm:hankprocess}. The process convergence of $w_p(t)$ for odd positive integer $p$ is given in Remark \ref{rem:hankprocess_wp_odd_proof}.

Now we collect an existing result on fluctuations of linear eigenvalue statistics of band Hankel matrices. We first introduce some notations to state the result.

We consider an input sequence $\{ x_i : i \geq 1 \} $ of independent real random variables which satisfy the following moment assumptions:
 \begin{equation}\label{eqn:condition}
\E[x_i]=0, \quad  \E[x_i^2]=1, \ \kappa=\E [x_{j}^{4}] \,\, \forall \ i \quad \text{and} \quad \sup _{j \geq 1} \E\left[\left|x_{j}\right|^{k}\right]=\alpha_{k}<\infty \quad \text { for } \quad k \geq 3.
\end{equation}  

along with $x_0 \equiv 0$. For a sequence $b_n$, with $b_n \rightarrow \infty$ and $b_n/n \rightarrow b$,  we define the band Hankel matrix $H_n$ as $H_n = P_n T_n$ where $T_n= (\delta_{|i-j| \leq b_n}x_{i-j})$. Furthermore, we define $A_n= H_n/\sqrt{b_n}$ and for each $p \in \N, w_p$ as

\begin{equation}\label{eqn:w_p}
w_p := \frac{\sqrt{b_n}}{n} \bigl\{ \Tr(A_n)^{p} - \E[\Tr(A_n)^{p}]\bigr\}. 
\end{equation}

 In \cite{liu2012fluctuations}, Liu et al. proved the following result on fluctuations of $w_p$ for even values of $p$.

\begin{result}  \label{result: even fluctuation}(Theorem 6.4, \cite{liu2012fluctuations})
Let $H_{n}$ be a real symmetric random
Hankel band matrix whose input sequence satisfies (\ref{eqn:condition}),
and $b_{n}/n \rightarrow b \in [0,1]$ with $b_{n}\tends \infty$ as $n\rightarrow
\infty$. 

 Then for every even integer $p\geq 2$, as $n\rightarrow
\infty$
\begin{align*}
w_{p} \stackrel{d}{\rightarrow}  N(0,\sigma_p^2), %\omega_{Q}\longrightarrow N(0,\sigma_Q^2)
\end{align*}
where $w_p$ as defined in (\ref{eqn:w_p}).  
Here $\sigma_p^{2}$ is appropriate constant, see \cite[Theorem 6.4]{liu2012fluctuations}. 
\end{result}

The following theorem provide the fluctuations of $w_p$ for odd value of $p$.

\begin{theorem}\label{thm:hank_odd_wp} 
Suppose  $H_n$ is a band Hankel matrix with an input sequence $\{x_i\}_{i\geq 1}$ of independent real random variables which satisfies (\ref{eqn:condition}) and $\{b_n\}$ as in Result \ref{result: even fluctuation}.
Then, for odd $p \geq 1$, as $n\rightarrow
\infty$
\begin{align*}
w_p \stackrel{d}{\rightarrow} 0.
\end{align*}
\end{theorem}

Now we briefly outline the rest of the paper. In Section \ref{preliminaries}, first we state trace formula of Hankel matrices and then some basic results of process convergence.
In Section \ref{sec:partition}, we introduce some standard partitions and integrals.
 In Section \ref{sec:process convergence}, we prove Theorem \ref{thm:hankprocess} by using moment method and some results on process convergence. Finally, we prove Theorem \ref{thm:hank_odd_wp} in Section \ref{sec:thm:hank_odd_wp}.

\section{Preliminaries}\label{preliminaries}

     Here we introduce some results and definitions which will be used in the proof of theorems.  First we state the trace formula for band Hankel matrices. 
\begin{result} [Lemma 6.3, \cite{liu2012fluctuations}] \label{res:trace_Hn} 
Suppose  $H_n$ is a band Hankel matrix with an input sequence $\{x_i\}_{i\geq 1}$. Then,
\begin{align*} \label{def:trace_Hn}
\Tr(H_n)^p  
  = \begin{cases} 
    \displaystyle \sum_{i=1}^{n} \sum_{j_{1}, \ldots, j_{p}=-b_{n}}^{b_{n}} \prod_{\ell=1}^{p} x_{j_{\ell}} \prod_{\ell=1}^{p} \chi_{[1, n]}\left(i-\sum_{q=1}^{\ell}(-1)^{q} j_{q}\right) \delta_{0, \sum_{q=1}^{p}(-1)^{q} j_{q}}, &p \quad \text{even}; \\
     \displaystyle  \sum_{i=1}^{n} \sum_{j_{1}, \ldots, j_{p}=-b_{n}}^{b_{n}} \prod_{\ell=1}^{p} x_{j_{\ell}} \prod_{\ell=1}^{p} \chi_{[1, n]}\left(i-\sum_{q=1}^{l}(-1)^{q} j_{q}\right) \delta_{2 i-1-n, \sum_{q=1}^{p}(-1)^{q} j_{q}}, & p \quad \text{odd}. 
   \end{cases}
\end{align*}
\end{result}

    Now we see some standard results on process convergence.

Let $( C_{\infty},\mathcal{C_{\infty}} )$ be a probability measure space, where $C_{\infty}:= C[0, \infty)$ is the space of all real-valued continuous functions on $[0, \infty)$ and $\mathcal{C_{\infty}}$ is a $\sig$-field generated by open sets of $C_{\infty}$. For more details about $( C_{\infty},\mathcal{C_{\infty}} )$ and $\sig$-field, see \cite{m&s_toeplitz_2020}.

\begin{definition}\label{processconvergence}
 Suppose $\{\boldsymbol{X_n}\}_{n\geq 1}$ = $\{ X_n(t); t \geq 0 \}_{n\geq 1}$ and $\boldsymbol{X}=\{X(t);t\ge 0\}$ are real-valued continuous process on $C[0, \infty)$.
 We say $\boldsymbol{X_n}$ converge to $\boldsymbol{X}$ \textit{weakly} or in \textit{distribution}, denote by $\boldsymbol{X_n} \stackrel{\mathcal D}{\rightarrow} \boldsymbol{X}$ if $\P_n \stackrel{\mathcal D}{\rightarrow} \P $, that is,

$$ \P_n f :=  \int_{C_{\infty}} f d \P_n \longrightarrow \P f := \int_{C_{\infty}} f d \P \ \mbox{ as }n\to\infty,$$
 for every bounded, continuous real-valued function $f$ on $C_{\infty}$, where $\P_n$ and $\P$ are the probability measures on $( C_{\infty},\mathcal{C_{\infty}} )$ induced by $\boldsymbol{X_n}$ and $\boldsymbol X$, respectively.
 \end{definition}
 
 Suppose $\{\boldsymbol{X_n}\}_{n\geq 1}$ = $\{ X_n(t); t \geq 0 \}_{n\geq 1}$ and $\boldsymbol{X}$ = $\{ X(t); t \geq 0 \}$ are sequences of continuous processes on $C[0, \infty)$.
Then  $\boldsymbol{X_n} \stackrel{\mathcal D}{\rightarrow} \boldsymbol{X}$ if and only if: \\
	\noindent \textbf{(i)} \textbf{Finite dimensional convergence:} Suppose $0<t_1<t_2 \ldots <t_r$. Then as $n \tends \infty$
	\begin{equation*} \label{eqn:finite_dim}
	(X_{n}(t_1), X_{n}(t_2), \ldots , X_{n}(t_r)) \stackrel{\mathcal D}{\rightarrow} (X(t_1), X(t_2), \ldots , X(t_r)),
	\end{equation*}	
	\vskip3pt
	\noindent \textbf{(ii)} \textbf{Tightness:} the sequence  $\{\boldsymbol{X_n}\}_{n\geq 1}$ is tight. 

Take $X_n(t)$ = $w_{p}(t)$ where $w_{p}(t)$ is as defined in (\ref{eqn:w_p(t)}). From the above discussion and \cite[Theorem I.4.3]{Ikeda1981}, it is clear that Theorem \ref{thm:hankprocess} follows from the following two propositions:
 
%From the above discussion on process convergence, it is clear that, 
%to prove Theorem \ref{thm:hankprocess}, it is sufficient to prove the following two propositions for the process $\{w_{p}(t);t \geq 0\}$: 
\begin{proposition} \label{pro:finite convergence}
	For each $ p\geq 2$, suppose $0<t_1<t_2 \ldots <t_r$. Then as $n \tends \infty$
	\begin{equation*}
	(w_{p}(t_1), w_{p}(t_2), \ldots , w_{p}(t_r)) \stackrel{\mathcal D}{\rightarrow} (W_{p}(t_1), W_{p}(t_2), \ldots , W_{p}(t_r)).
	\end{equation*}	
\end{proposition}

\begin{proposition} \label{pro:tight}
	For each $p \geq 2$, there exists positive constants $M$ and $\gamma$ such that
	\begin{equation*} \label{eq:tight3}
	\E{|w_{p}(0)|^ \gamma } \leq M \ \ \ \forall \ n \in \mathbb{N}, 	 
	\end{equation*}
	and there exists positive constants $\alpha, \ \beta $ and $M_T$, $T= 1, 2, \ldots,$ such that 
	\begin{equation*} \label{eq:tight4}
	\E{|w_{p}(t)- w_{p}(s)|^ \alpha } \leq M_T |t-s|^ {1+ \beta} \ \ \ \forall \ n \in \mathbb{N} \ \mbox{and } t,s \in[0,\ T],(T= 1, 2, \ldots). 
	\end{equation*}
	Then $\{ w_{p}(t) ; t \geq 0 \}$ is tight.
\end{proposition} 

  \section{Some partitions and integrals}\label{sec:partition}

In this section, we introduce some partitions and integrals which will be required for the proof of Theorem \ref{thm:hankprocess}.
\begin{definition}\label{def: matched element}
 Let $J_1 \in \mathbb{R}^p, J_2 \in \mathbb{R}^q$, with $J_1=(j_1, \ldots , j_p)$ and $J_2 = (j_{p+1}, \ldots , j_{p+q})$. Consider two components $j_{k_1}$ and $j_{k_2}$ with $j_{k_1}= j_{k_2}$. Then
 \begin{enumerate}
 \item[(i)]   $j_{k_1}$ and $j_{k_2}$ are said to be \textit{self-matched in} $J_1$ if $1\leq k_1,k_2 \leq p$.
 \item[(ii)] $j_{k_1}$ and $j_{k_2}$ are said to be \textit{self-matched in} $J_2$ if $p+1\leq k_1,k_2 \leq p+q$.
 \item[(iii)] $j_{k_1}$ and $j_{k_2}$ are said to be \textit{cross-matched} if $1\leq k_1 \leq p< k_2 \leq p+q$ or $1\leq k_2 \leq p< k_1 \leq p+q$.
 \end{enumerate}
\end{definition} 

\begin{definition}
   Consider the set $[n]=\{ 1,2 ,\ldots , n\}$. A partition $\pi$ of $[n]$ is called a \textit{pair-partition} if each equivalency class of $\pi$ has exactly two elements. If $i,j$ belongs to the same equivalency class of $\pi$, we write $i \sim_{\pi} j$. The set of all pair-partitions of $[n]$ is denoted by $\PP_2(n)$. Clearly, $\PP_2(n)= \emptyset$ for odd n.
\end{definition}

\begin{definition}
    Consider an even number $n$. A pair-partition $\pi \in \PP_2(n)$ is called an \textit{odd-even pair partition} if for all $i,j$ such that $i \sim_{\pi} j$ and $i \neq j$, one of $\{ i,j \}$ is odd and the other is even. The set of all odd-even pair partitions of $[n]$ is denoted by $\Delta_2(n)$. For odd $n$, $\Delta_2(n)= \emptyset$.
    %to be the null set.
\end{definition}

\begin{definition}
    Consider $p,q \in \N$ such that $p+q$ is even. We define the set $\Delta_2(p,q) \subset \PP_2(p+q)$ in the following way:
    \begin{enumerate}[(i)]
        \item Any $\pi \in \Delta_2(p+q)$ is an odd-even pair partition.
        \item For every $\pi \in \Delta_2(p+q),$ there exists $i,j$ such that $i \leq p <j$, such that $i \sim_{\pi} j$.
    \end{enumerate}
For $p,q$ such that $p+q$ is odd, $\Delta_2(p+q)$ is taken to be the empty set.
\end{definition}

Consider, for fixed natural numbers $p$ and $q$, $J=(j_1,j_2, \ldots , j_{p+q}) \in \mathbb{R}^{p+q}$. Then $J$ can be thought of as a 2-tuple $J=(J_1,J_2)$ where $J_1=(j_1,j_2, \ldots , j_p)$ and $J_2=(j_{p+1},j_{p+2}, \ldots , j_{p+q})$. Thus, taking cue from Definition \ref{def: matched element}, we can introduce the terms self-matched and cross-matched for any vector $J$. 
A closer observation yields that the concept of self-matching (and cross-matching) depends only on the relationship between the components and not on the numerical value of the pair. Therefore the concept has a natural extension to any partition $\pi$ of $[p+q]$. 

\begin{definition}
    Consider $\PP_2(p,q)$, the subset of $\PP_2(p+q)$ of elements having at least one cross-matched pair. Consider the set of all $\pi \in \PP_2(p,q)$ such that
    \begin{enumerate}[(i)]
        \item every self-matched pair of $\pi$ is an odd-even pair,
        \item every cross-matched pair of $\pi$ is either an even-even pair or an odd-odd pair.
    \end{enumerate}
    The set of all such $\pi$ is denoted by $\tilde{\Delta}_2(p,q)$.
\end{definition}

We also define the following important partition of $[p+q]$ for natural numbers $p$ and $q$.
\begin{definition}
    Let $\pi =\{ V_1 \ldots, V_r \}$ be a partition of $[p+q]$. We say $\pi \in \Delta_{2,4}(p,q)$ if the following conditions are satisfied:
    \begin{enumerate}[(i)]
        \item there exists such that $|V_i|=4$ and $|V_j|=2$ for all $j\neq i$,
        \item for each $j \neq i$, one element of $V_j$ is odd and the other is even,
        \item for each $j \neq i$, $V_{j} \subset\{1, \ldots, p\}$ or $V_{j} \subset\{p+1, \ldots, p+q\}$,
        \item two elements of $V_i$ come from $\{1, \ldots, p\}$ and they form an odd-even pair. The other two elements coming from $\{p+1, \ldots, p+q\}$ also forms an odd-even pair.
    \end{enumerate}
\end{definition}
 
For a partition $\pi$ and a set of unknowns $\{x_i \}$, we can construct a new set of unknowns $\{y_i \}$ by defining $y_i = x_{\pi(i)}$. Now we define some integrals associated with the partition.

 \begin{definition} \label{def:f_I}
 
    \begin{enumerate}[(i)]
        \item  For $\pi \in \Delta_2(p,q)$, define 
            \begin{align*}
            f_{I}^{-}\left(\pi\right)=& \int_{[0,1]^{2}} d y_{0} d x_{0} \int_{\left[-1, 1\right]^{\frac{p+q-2}{2}}} d x_{1} dx_{2} \cdots d x_{(p+q) / 2} \\
            & \times \delta\left( \sum_{i=1}^{p}(-1)^i y_{i}\right) \prod_{j=1}^{p} I_{[0,1]}\left(x_{0}+b \sum_{i=1}^{j} (-1)^i y_{i}\right) \prod_{j^{\prime}=p+1}^{p+q} I_{[0,1]}\left(y_{0}+b \sum_{i=p+1}^{j^{\prime}}  (-1)^i y_{i}\right).
            \end{align*}
        \item For $\pi \in \tilde{\Delta}_2(p,q)$, define 
            \begin{align*}
            f_{I}^{+}\left(\pi \right)=& \int_{[0,1]^{2}} d y_{0} d x_{0} \int_{\left[-1, 1\right]^{\frac{p+q-2}{2}}} d x_1 d x_2  \cdots d x_{(p+q) / 2} \\
            & \times \delta\left(\sum_{i=1}^{p} (-1)^i y_{i}\right) \prod_{j=1}^{p} I_{[0,1]}\left(x_{0}+b \sum_{i=1}^{j}(-1)^i y_{i}\right) \prod_{j^{\prime}=p+1}^{p+q} I_{[0,1]}\left(y_{0}-b \sum_{i=p+1}^{j^{\prime}} (-1)^i y_{i}\right).
            \end{align*}
        \item For $\pi \in \Delta_2(p,q)$, define 
            \begin{align*}
            f_{I I}^{-}\left(\pi \right)=& \int_{[0,1]^{2}} d y_{0} d x_{0} \int_{\left[-1, 1\right]^{\frac{p+q-2}{2}}} d x_{1} \cdots d x_{\frac{p+q}{2}-1} \\
            & \times \prod_{j=1}^{p} I_{[0,1]}\left(x_{0}+b \sum_{i=1}^{j} (-1)^i y_{i}\right) \prod_{j^{\prime}=p+1}^{p+q} I_{[0,1]}\left(y_{0}+b \sum_{i=p+1}^{j^{\prime}} (-1)^i y_{i}\right),
            \end{align*}
            \begin{align*}
            f_{I I}^{+}(\pi)=& \int_{[0,1]^{2}} d y_{0} d x_{0} \int_{\left[-1, 1 \right]^{p / 2}} d x_{1} \cdots d x_{\frac{p+q}{2}-1} \\
            & \times \prod_{j=1}^{p} I_{[0,1]}\left(x_{0}+b \sum_{i=1}^{j} (-1)^i y_{i}\right) \prod_{j^{\prime}=p+1}^{p+q} I_{[0,1]}\left(y_{0}-b \sum_{i=p+1}^{j^{\prime}} (-1)^i y_{i}\right).
            \end{align*}
    \end{enumerate}
 \end{definition}

  \section{Proof of Theorem \ref{thm:hankprocess}}\label{sec:process convergence}
First recall that Propositions \ref{pro:finite convergence} and \ref{pro:tight} yield Theorem \ref{thm:hankprocess}. In Subsection \ref{subsec:finite dim}, we prove Proposition \ref{pro:finite convergence} and in \ref{subsec:tight}, we prove Proposition \ref{pro:tight}. Throughout this section we assume that $b_n=o(n)$, which gives $b= \lim_{n \rightarrow \infty } b_n/n =0$.
\subsection{ Proof of Proposition \ref{pro:finite convergence}}\label{subsec:finite dim}

 We first start with the following two lemmata which will used in the proof of Proposition \ref{pro:finite convergence}.

\begin{lemma}\label{lemma: Process cov}
Suppose $0<t_1 \leq t_2$, $p$ and $q$ are even natural numbers. For $j \in \mathbb{Z}$, let $u_j=a_j(t_1)$ and $v_j=a_j(t_2)-a_j(t_1)$. Then
\begin{align} \label{eqn:E[wp,wq_hp}
    \displaystyle \E \left[w_{p}\left(t_{1}\right) w_{q}\left(t_{2}\right)\right] & = b_{n}^{-\frac{p+q}{2}+1} \sum_{r=0}^{q}{q \choose r}  \sum_{ J=(j_{1}, \ldots, j_{p}) \in A_p , \atop J'  = (j_{1}^{\prime}, \ldots, j_{q}^{\prime}) \in A_q}
    \left( \E [u_{j_{1}} \cdots u_{j_{p}} u_{j_{i}^{\prime}} \cdots u_{j_{r}^{\prime}} v_{j_{r+1}^{\prime}} \cdots v_{j_{q}^{\prime}}]\right. \nonumber\\
& \qquad \displaystyle  \left. - \E [u_{j_{1}} \cdots u_{j_{p}}] \E[u_{j_{i}^{\prime}} \cdots u_{j_{r}^{\prime}} v_{j_{r+1}^{\prime}} \cdots v_{j_{q}^{\prime}}]\right) +o(1), 
%\delta\left(\sum_{k=1}^{p}(-1)^k j_k\right) \delta\left(\sum_{k=1}^{q}(-1)^k {j_k}^{\prime}\right) +o(1).
\end{align}
where for each $p \in \mathbb{N},$ $A_{p}$ is defined as
   \begin{equation}\label{def:A_{p}}
    A_{p}=\left\{\left(j_{1}, \ldots, j_{p}\right) \in\left\{\pm 1, \ldots, \pm b_{n}\right\}^{p}: \sum_{k=1}^{p} (-1)^{k} j_{k}=0\right\}.
\end{equation}
\end{lemma}

\begin{proof}[Proof] The proof is similar to the proof of Lemma $4.1$ of \cite{li_sun_2015}. So, we skip the details. 
\end{proof}
\begin{lemma}\label{lemma: process cov2}
If $0 < t_1\leq t_2$ and $p, q \geq 2$ be even natural numbers, then
\begin{equation} \label{eqn:Cov_wp_hp}
\lim_{n \to \infty} \E[w_p(t_1)w_q(t_2)]= 2^{\frac{p+q}{2}} \sum_{r=2,4, \ldots, q}{q \choose r}t_{1}^{\frac{p+r}{2}} \left(t_{2}-t_{1}\right)^{\frac{q-r}{2}}  R(p, r) \Gamma\left(\frac{q-r}{2}+1 \right),
 \end{equation}
 where $R(p,r)=\left|\Delta_{2}(p, q)\right|+\left|\tilde{\Delta}_{2}(p, r)\right|+2\left|\Delta_{2,4}(p, r)\right|$ and $\Gamma$ denotes the Gamma function.
\end{lemma}
\begin{proof}

 First we define $U:=A_n(t_1)$ and $V:=A_n(t_2)-A_n(t_1)$. Further for $0 \leq r \leq q$, we define
\begin{align} \label{eqn:J,J',U_j}
 J=&(j_1,j_2, \ldots j_p), \quad J^{\prime} = (j_1^{\prime},j_2^{\prime},\ldots ,j_q^{\prime}),	\quad
J_{1}^{\prime}=\left(j_{1}^{\prime}, \ldots, j_{r}^{\prime}\right),  \quad  J_{2}^{\prime}=\left(j_{r+1}^{\prime}, \ldots, j_{q}^{\prime}\right), \\  
u_{\mathrm{J}} = & u_{j_{1}} \cdots u_{j_{p}}, \quad  u_{\mathrm{J}_{1}^{\prime}}=u_{j_{1}^{\prime}} \cdots u_{j_{r}^{\prime}} \ \mbox{ and } \ v_{\mathrm{J}_{2}^{\prime}}=v_{j_{r+1}^{\prime}} \cdots v_{j_{q}^{\prime}}. \nonumber
\end{align} 	

Now with the above notations, (\ref{eqn:E[wp,wq_hp}) can be written as
\begin{align}\label{eqn: cov wp(t) wq(t)}
\E\left[w_{p}\left(t_{1}\right) w_{q}\left(t_{2}\right)\right] %&=b_{n}^{-\frac{p+q}{2}+1} \sum_{r=0}^{q}{q \choose r} \sum_{J, J^{\prime}}\left(\left(\E\left[u_{J} u_{J_{1}^{\prime}} v_{J_{2}^{\prime}}\right]-\E\left[u_{J}\right] \E\left[u_{J_{1}^{\prime}} v_{J_{2}^{\prime}}\right]\right)\right)+o(1), \nonumber \\
&=b_{n}^{-\frac{p+q}{2}+1} \sum_{r=0}^{q}
{q \choose r}
 \sum_{J, J^{\prime}}\Big\{\left(\E\left[u_{J} u_{J_{1}^{\prime}}\right]-\E\left[u_{J}\right] \E\left[u_{J_{1}^{\prime}}\right]\right) \E\left[v_{J_{2}^{\prime}}\right]\Big\}+o(1),
\end{align}
where $J \in A_{p}$ and $J^{\prime} \in A_{q}$ . Now for a fixed $r$, $0 \leq r \leq q$, consider 
\begin{equation}\label{eqn:J,J_1,J_2}
   b_{n}^{-\frac{p+q}{2}+1} \sum_{J, J^{\prime}}\Big\{\left(\E\left[u_{J} u_{J_{1}^{\prime}}\right]-\E\left[u_{J}\right] \E\left[u_{J_{1}^{\prime}}\right]\right) \E\left[v_{J_{2}^{\prime}}\right]\Big\}. 
\end{equation}

Consider the term
   $\left(\E[u_{J} u_{J_{1}^{\prime}}]-\E[u_{J}] \E[u_{J_{1}^{\prime}}]\right) \E[v_{J_{2}^{\prime}}]$. Due to the independence of random variables $\{u_i\}$ and $\{ v_i \}$, this term is non-zero only when the following three conditions are satisfied:
   
   \begin{enumerate}[(a)]
      \item there exists at least one common element between $S_J$ and $S_{J_1^{\prime}}$, that is, $S_J \cap S_{J_1^{\prime}} \neq \emptyset$,
        \item every element of $S_J \cup S_{J_1^{\prime}}$ has cardinality at least 2, and
         \item every element of $S_{J_2^{\prime}}$ has cardinality at least 2,
   \end{enumerate}
   where for a vector $J=\left(j_{1}, j_{2}, \ldots, j_{p}\right) \in A_{p},$ 
the multi-set $S_{J}$ is defined as
\begin{equation} \label{def:S_J_Hn}
  S_{J}=\left\{ j_{1},j_{2}, \ldots,j_{p} \right\}. 
\end{equation}
 Now consider the map
  $L:\{ \pm 1, \pm 2, \ldots \pm b_n \}^p \times \{ \pm 1, \pm 2, \ldots \pm b_n \}^q \rightarrow \{ \pm 1, \pm 2, \ldots \pm b_n \}^{p+q-2} $
  defined in the following way:
  Let $j_u \in J$ be the first cross-matched element. If $(-1)^u j_u=(-1)^v j_v^{\prime}$ for some $v$, define $L \in \{ \pm 1, \pm 2, \ldots \pm b_n \}^{p+q-2} $ by
\begin{align*}
l_{1} & =j_{1}, \ldots, l_{u-1}=j_{u-1}, l_{u}=j_{1}^{\prime}, \ldots, l_{u+v-2}=-j_{v-1}^{\prime}, \\
l_{u+v-1} & = -j_{v+1}^{\prime}, \ldots, l_{u+q-2}=-j_{q}^{\prime}, l_{u+q-1}=j_{u+1}, \ldots, l_{p+q-2}=j_{p}.
\end{align*}

And if for the first cross-matched element $j_u$, $(-1)^u j_u=-(-1)^v j_v^{\prime}$, then construct $L$ by starting from $(J,-J^{\prime})$.
This process of constructing $L$ from $(J,J^{\prime})$ is called a \textit{reduction} step and is denoted by $L=J\bigvee_{j_u}J^{\prime}$. 

Notice that if $J \in A_{p}$ and $J^{\prime} \in A_{q}$, then $L \in A_{p+q-2}$. Furthermore, for every $L \in A_{p+q-2}$, we can find at most $pq$ pairs $(J,J^{\prime}) \in A_{p} \times A_{q}$ such that $L=J\bigvee_{j_u}J^{\prime}$. This is because when $p-1$ elements in $J$ (or similarly $q-1$ elements in $J^{\prime}$) are fixed, the value of the left out element becomes automatically fixed.
\begin{localclaim}\label{claim: multiplicity 2}
Non-zero contribution for (\ref{eqn: cov wp(t) wq(t)}) are due to those $L$ such that each element of $S_L$ have multiplicity 2.
\end{localclaim}
Consider $L \in A_{p+q-2}$ obtained from $(J,J^{\prime})$ satisfying the above conditions (a) - (c). Notice that for any such $L$, there can exist at most one element in $S_L$ of multiplicity 1. Since $L \in A_{p+q-2}$, the element of multiplicity 1 is determined by other elements of $L$. Furthermore, this case is only possible when the cross-matched element has multiplicity 3 in $S_J \cup S_{J^{\prime}}$. Thus the number of all $(J,J^{\prime})$ which has such $L$ as its image is $O(b_n^{\frac{p+q-3}{2}})$ which implies that the contribution of such terms to $\E[w_p w_q]$ is $O(b_n^{-\frac{1}{2}})$. It is also easy to see that the sum of all terms $(J,J^{\prime})$ such that an $S_L$ has an element of multiplicity 3, leads to zero contribution in (\ref{eqn:J,J_1,J_2}). Hence the claim.

The claim along with constraints (b) and (c) implies that each element in $S_J \cup S_{J_1^{\prime}}$ and $S_{J_2^{\prime}}$ are of cardinality 2 and that $S_{J_1^{\prime}}$ and $S_{J_2^{\prime}}$ have no elements in common. Furthermore, the contribution is maximum when $J \in A_p, J_1^{\prime} \in A_r$ and $J_2\in A_{q-r}$.
Thus (\ref{eqn:J,J_1,J_2}) becomes
\begin{equation*}
\noindent
    \left(b_{n}^{-\frac{p+r}{2}+1} \sum_{J, \mathrm{J}_{1}^{\prime}}\left(\E\left[u_{{J}} u_{{J}_{1}^{\prime}}\right]-\E\left[u_{{J}}\right] \E\left[u_{{J}_{1}^{\prime}}\right]\right)\right)\left(b_{n}^{-\frac{q-r}{2}} \sum_{{J}_{2}^{\prime}} \E\left[v_{{J}_{2}^{\prime}}\right]\right)+o(1),
\end{equation*}
where $  J \in A_p$, $J_1^{\prime} \in A_r$ and $J_2^{\prime} \in A_{q-r}$. Notice that when $r$ is odd, the pair partition mentioned above cannot happen. Hence non-zero contribution occurs only for even $r$.

Claim \ref{claim: multiplicity 2} also implies that it is sufficient to fix a particular $\pi \in \PP_{2}(p+q-2)$ and count the number of $L \in A_{p+q-2}$ corresponding to $\pi$. Since each element in $S_L$ is repeated exactly twice, the number of $L$ that correspond to a particular $\pi$ is of order at most $O(b_n^{\frac{p+q-2}{2}})$. Now since the number of pre-images of a particular $L$ is at most $pq$, only those $\pi$ from which  $O(b_n^{\frac{p+q-2}{2}})$ many $L'$s can be generated contribute to $\lim_{n \rightarrow \infty} \E[w_p w_q]$, which are exactly those $\pi \in \Delta_2(p+q-2)$. 

Thus we obtain for $r$ even, 
\begin{align}\label{eqn U_j U_J_1}
 &  \lim_{n \rightarrow \infty}  b_{n}^{-\frac{p+r}{2}+1} \sum_{{J}, {J}_{1}^{\prime}}\left(\E\left[u_{{J}} u_{{J}_{1}^{\prime}}\right]-\E\left[u_{{J}}\right] \E\left[u_{{J}_{1}^{\prime}}\right] \right) \nonumber \\ 
 & =\left(\E\left[u_{i}^{2}\right]\right)^{\frac{p+r}{2}}\left[\sum_{\pi \in \Delta_{2}(p, r)}f_{I}^{-}(\pi)+\sum_{\pi \in \tilde{\Delta}_{2}(p,r)}f_{I}^{+}(\pi)\right] \nonumber \\
  &  \qquad +\left(\left(\E\left[u_{i}^{2}\right]\right)^{\frac{p+r-4}{2}} \E\left[u_{i}^{4}\right]-\left(\E\left[u_{i}^{2}\right]\right)^{\frac{p+r}{2}}\right)  \times\left[\sum_{\pi \in {\Delta_{2,4}}(p, r)}f_{I I}^{-}(\pi)+\sum_{\pi \in {\Delta}_{2,4}(p, r)}f_{I I}^{+}(\pi)\right],
\end{align}
where $f_{I}^{-}(\pi), f_{I}^{+}(\pi), f_{I I}^{-}(\pi)$ and $ f_{I I}^{+}(\pi)$ are as in Definition \ref{def:f_I}.
Since we have $a_i(t)$ to be Brownian motion entries, $\E[u_i^2]=t_1$ and $\E[u_i^4]=3t_1^2$ for all $i \neq 0$. Since $b=0$ as $b_n=o(n)$, we get
$$f_{I}^{-}(\pi)=f_{I}^{+}(\pi)=f_{I I}^{-}(\pi)=f_{I I}^{+}(\pi)=2^{\frac{p+r}{2}-1}.$$
Therefore (\ref{eqn U_j U_J_1}) will be
\begin{align}\label{eqn : simp uJ uJ1}
  %  \lim_{n \rightarrow \infty} & b_{n}^{-\frac{p+r}{2}+1} \sum_{{J}, {J}_{1}^{\prime}}\left(\E\left[u_{{J}} u_{{J}_{1}^{\prime}}\right]-\E\left[u_{{J}}\right] \E\left[u_{{J}_{1}^{\prime}}\right] \right) \nonumber\\
     \begin{cases}
    2^{\frac{p+r}{2}} t_{1}^{\frac{p+r}{2}}\left(\left| \Delta_2(p,q )\right|+ \left|\tilde{\Delta}_{2}(p, r)\right|+2\left|\Delta_{2,4}(p, r)\right|\right) &\text{when} \hspace{2mm} r \hspace{2mm} \text{is even,}\\
    0 & \text{otherwise}.
    \end{cases}
\end{align}
Now we proceed to find $b_{n}^{-\frac{q-r}{2}} \sum_{J_{2}^{\prime}} \E\left[v_{J_{2}^{\prime}}\right]$. For that consider a random band Hankel matrix $H_n$ with input variables as $\frac{v_{i-j}}{\sqrt{t_2-t_1}}$ and scaling as $\frac{1}{\sqrt{b_n}}$. Then

\begin{align*}
    &\frac{1}{n} \E\left[\Tr H_{n}^{q-r}\right]
    =\frac{1}{n\left(2 b_{n}\right)^{\frac{q-r}{2}}} \sum_{i=1}^{n} \sum_{j_{1}, \ldots, j_{q-r}=-b_{n}}^{b_{n}}  \E\left[w_{j_{1}} \cdots w_{j_{q-r}}\right] \prod_{s=1}^{q-r} I_{[1, n]}\left(i-\sum_{z=1}^{s} (-1)^z j_{z}\right) \delta\left(\sum_{z=1}^{q-r} (-1)^z j_{z}\right).
\end{align*}

Since $b_n=o(n)$, it is easy to see that $\prod_{s=1}^{q-r} I_{[1, n]}\left(i-\sum_{z=1}^{s} (-1)^z j_{z}\right)$ converges to 1 as $n$ tends to infinity, for all choices of $j_1, j_2, \ldots ,j_{q-r}$. Therefore using dominated convergence theorem, we get

\begin{align*}
  \lim_{n \rightarrow \infty}  \frac{1}{n} \E\left[\Tr H_{n}^{q-r}\right] &= \lim_{n \rightarrow \infty}  \frac{1}{\left(2 b_{n}\right)^{\frac{q-r}{2}}} \sum_{j_{1}, \ldots, j_{q-r}=-b_{n}}^{b_{n}} \E\left[w_{j_{1}} \cdots w_{j_{q-r}}\right] \delta\left(\sum_{z=1}^{q-r} (-1)^z j_{z}\right)+o(1), \\
 % &=  \lim_{n \rightarrow \infty} \frac{1}{\left[2 b_n \left(t_{2}-t_{1}\right)\right]^{\frac{q-r}{2}}} 
  %  \sum_{j_{1}, \ldots, j_{q-r}=-b_{n}}^{b_{n}} \E\left[v_{j_{1}} \ldots v_{j_{q-r}}\right] \delta\left(\sum_{z=1}^{q-r} (-1)^z j_{z}\right)+o(1), \\
  &=  \lim_{n \rightarrow \infty} \frac{1}{\left[2\left(t_{2}-t_{1}\right)\right]^{\frac{q-r}{2}}} \frac{1}{\left(b_{n}\right)^{\frac{q-r}{2}}} \sum_{J_{2}^{\prime}} \E\left[v_{J_{2}^{\prime}}\right]+o(1),
\end{align*}

where $J_{2}^{\prime}, v_{J_{2}^{\prime}}$ are as given in (\ref{eqn:J,J',U_j}) and $J_{2}^{\prime} \in A_{q-r}$.  By \cite[Theorem 3.2]{liu_wang2011}, the empirical measure of $H_n$ converge to the distribution given by $f(x)=|x| \exp \left(-x^{2}\right)$. Hence we get that $\lim _{n \rightarrow \infty} \frac{1}{n} \E\left[\Tr H_{n}^{q-r}\right]$ equals to $\Gamma(\frac{q-r}{2}+1)$ when $r$ is even and 0 when $r$ is odd, where $\Gamma$ denote the Gamma function.
Thus
\begin{equation}\label{eqn: simp v_J2}
    \lim _{n \rightarrow \infty} b_{n}^{-\frac{q-r}{2}} \sum_{J_{2}^{\prime}} \E\left[v_{J_{2}^{\prime}}\right]=\left\{\begin{array}{ll}
0 & \text { if } r \text { is odd, } \\
\Gamma(\frac{q-r}{2}+1)\left(t_{2}-t_{1}\right)^{\frac{q-r}{2}} 2^{\frac{q-r}{2}} & \text { if } r \text { is even. }
\end{array}\right.
\end{equation}

Now using (\ref{eqn : simp uJ uJ1}) and (\ref{eqn: simp v_J2}) in (\ref{eqn: cov wp(t) wq(t)}), we get that for $p,q$ even
\begin{align*}
   & \lim _{n \rightarrow \infty} \E\left[w_{p}\left(t_{1}\right) w_{q}\left(t_{2}\right)\right] \\
  % & \quad =\lim _{n \rightarrow \infty} b_{n}^{-\frac{p+q}{2}+1} \sum_{r=0}^{q}{q \choose r} \sum_{J, J^{\prime}}\left(\left(\E\left[u_{J} u_{J_{1}^{\prime}}\right]-\E\left[u_{J}\right] \E\left[u_{J_{1}^{\prime}}\right]\right) \E\left[v_{J_{2}^{\prime}}\right]\right)\\
   &= \sum_{r=2,4, \ldots, q}{q \choose r} t_{1}^{\frac{p+r}{2}} 2^{\frac{p+q}{2}}\left(\left| \Delta_2(p,q )\right|+ \left|\tilde{\Delta}_{2}(p, r)\right|+2\left|\Delta_{2,4}(p, r)\right|+\right)\Gamma\left(\frac{q-r}{2}+1\right)\left(t_{2}-t_{1}\right)^{\frac{q-r}{2}} \\
   &=  2^{\frac{p+q}{2}} \sum_{r=2,4, \ldots, q}{q \choose r} t_{1}^{\frac{p+r}{2}} \left(t_{2}-t_{1}\right)^{\frac{q-r}{2}} R(p,r)\Gamma\left(\frac{q-r}{2}+1\right).
\end{align*}
This completes the proof of the lemma.
\end{proof}

Now we prove Proposition \ref{pro:finite convergence} with the help of above lemmata.
\begin{proof}[Proof of Proposition \ref{pro:finite convergence}] 
 We use Cram\'er-Wold device and method of moment to prove  Proposition \ref{pro:finite convergence}. It is enough to show that, for $0<t_1 \leq t_2 \leq \cdots \leq t_\ell $ and for even integers $p_1, p_2, \ldots , p_\ell \geq 2$,
% With the above lemma, we are equipped to prove Theorem 7. We will only be giving a short argument since, the same argument is dealt in detail in the next section:
\begin{equation} \label{eqn:E_w_pi_hp}
\lim_{n\to\infty}\E[w_{p_1}(t_1)w_{p_2}(t_2) \cdots w_{p_\ell}(t_\ell)]=\E[W_{p_1}(t_1)W_{p_2}(t_2) \cdots W_{p_\ell}(t_\ell)],
\end{equation} 
where $\{W_p(t)\}_{p\geq 2}$ is a centered Gaussian family with covariance as in (\ref{eqn:W_p(t):cov_hp}). Note that the proof of (\ref{eqn:E_w_pi_hp}) goes similar to the proof of Theorem 1.4 of \cite{li_sun_2015}, once we get the covariance formula, (\ref{eqn:Cov_wp_hp}). So with the helps of above lemmata and by similar arguments of Theorem 1.4 of \cite{li_sun_2015}, we get the result. We skip the details here.
\end{proof}    

\subsection{ Proof of Proposition \ref{pro:tight}}\label{subsec:tight}

The idea of the proof of Proposition \ref{pro:tight} is similar to the proof of Proposition 10 in \cite{m&s_toeplitz_2020}. Here we outline only the main steps, for the details, see the proof of Proposition 10 of \cite{m&s_toeplitz_2020}.  
\begin{proof}[Proof of Proposition \ref{pro:tight}] 
 
We  prove Proposition \ref{pro:tight} for $\alpha = 4$ and $\beta = 1$. 
Suppose $p \geq 2$ is an even fixed positive integer and $t,s \in[0,T] $, for some fixed $T \in \mathbb{N}$. Then from (\ref{eqn:w_p(t)}),
\begin{align} \label{eqn:w_p(t-s)}
 w_{p}(t) - w_{p}(s) &= \frac{\sqrt{b_n}}{n} \big[ \Tr(A_n(t))^{p} - \Tr(A_n(s))^{p} - \E[\Tr(A_n(t))^{p} - \Tr(A_n(s))^{p}]\big].
%  \E[w_{2p}(t) -w_{p}(s)]^4 &= \frac{1}{n^2} \E \big[ \Tr(T_n(t))^{2p} - \Tr(RC_n(s))^{2p} - \E[\Tr(RC_n(t))^{2p} - \Tr(RC_n(s))^{2p}]\big]^4.
\end{align}
 For $0< s <t $, using binomial expansion, we get 
\begin{equation} \label{eqn:trace t-s}
\Tr (A_n(t))^{p} - \Tr (A_n(s))^{p}  = \Tr[ (A'_n(t-s) )^{p}] + \sum_{d=1}^{p-1} \binom{p}{d} \Tr [(A'_n(t-s))^d (A_n(s))^{p-d}],
\end{equation}
where $A'_n(t-s)$ is a band Hankel matrix with entries $\{ a_n(t) - a_n(s)\}_{n \geq 0}$. 
Now by using the trace formula of band Hankel matrices from Result \ref{res:trace_Hn}, we get 

\begin{align} \label{eqn:tracemultply}
	\Tr[ (A'_n(t-s) )^{p}] & =\frac{1}{ {b_n}^{\frac{p}{2}} } \sum_{i=1}^n \sum_{A_{p}} a'_{J_{p}}(t-s) I(i, J_p),  \\
	\Tr [(A'_n(t-s))^d (A_n(s))^{p-d}] & = \frac{1}{ {b_n}^{\frac{p}{2}} } \sum_{i=1}^n \sum_{A_{p}}  a'_{J_{d}}(t-s) a_{J_{p-d}}(s)  I(i, J_p),  \nonumber
	\end{align}
where 
\begin{align*}
a'_{J_{p}}(t-s) & = (a_{j_1}(t)-a_{j_1}(s)) (a_{j_2}(t) - a_{j_2}(s)) \cdots (a_{j_{p}}(t) - a_{j_{p}}(s)), \\
a'_{J_{d}}(t-s) & = (a_{j_1}(t)-a_{j_1}(s)) (a_{j_2}(t) - a_{j_2}(s)) \cdots (a_{j_{d}}(t) - a_{j_{d}}(s)), \\
a_{J_{p-d}}(s) & = a_{j_{d+1}}(s) a_{j_{d+2}}(s)\cdots a_{j_{p}}(s).
\end{align*}
Note that, $a_n(t) -a_n(s)$ has same distribution as $a_n(t-s)$ and  $a_n(t)$ has same distribution as $N(0, t)$. Therefore by using (\ref{eqn:tracemultply}) in (\ref{eqn:trace t-s}), we get 
\begin{align} \label{eqn:Z_p}
&  \Tr (A_n(t))^{p}- \Tr (A_n(s))^{p} \nonumber \\
  & \distas{D}  \frac{1}{ {b_n}^{\frac{p}{2}} } \sum_{i=1}^n \Big[ \sum_{A_{p}} \Big( (t-s)^{\frac{p}{2}} x_{j_1} x_{j_2} \cdots x_{j_{p}} + \sum_{d=1}^{p-1} \binom{p}{d}  (t-s)^{\frac{d}{2}} {s}^{\frac{p-d}{2}} x_{j_1} x_{j_2} \cdots x_{j_{d}}  y_{j_{d+1}} y_{j_{d+2}} \cdots y_{j_{p}}  \Big) \Big]  I(i, J_p) \nonumber \\
% & = \sqrt{\frac{t-s} { {b_n}^{p} } }\sum_{i=1}^n \Big[ \sum_{A_{p}} \Big( (t-s)^{\frac{p-1}{2}} X_{J_{p}} + \sum_{d=1}^{p-1} \binom{p}{d}  (t-s)^{\frac{d-1}{2}} {s}^{\frac{p-d}{2}} X_{J_{d}}  Y_{J_{p-d}} \Big) \Big] I(i, J_p) \nonumber \\
 & = \sqrt{\frac{t-s} { {b_n}^{p} } }\sum_{i=1}^n  \sum_{A_{p}} Z_{J_{p}} I(i, J_p), \mbox{ say},
\end{align} 
where $\{x_{i}\}$ and $\{y_{i}\}$ are independent normal random variables with mean zero, and variances $(t-s)$ and $s$, respectively. If $x_1$ and $x_2$ have same distribution, then we denote it as $x_1 \distas{D} x_2$. 
Finally, by using (\ref{eqn:Z_p}) in (\ref{eqn:w_p(t-s)}), we get
\begin{align} \label{eqn:w_p^4}
\E[w_{p}(t) -w_{p}(s)]^4 & = \frac{(t-s)^2}{n^4 {b_n}^{2p-2}} \sum_{i_1, i_2, i_3, i_4=1}^n  \sum_{A_{p}, A_{p}, A_{p}, A_{p}} \E \Big[ \prod_{k=1}^4 (Z_{J^k_{p}} - \E Z_{J^k_{p}})  I(i_k, J^k_p) \Big], 
\end{align}
where $J^k_{p}$ are vectors from $A_{p}$, for each $k=1, 2, 3, 4$. 

We introduce the terminology: $J^k_{p}$ and $J^\ell_{p}$ are connected if $S_{J^k_{p}} \cap S_{J^\ell_{p}} \neq \emptyset$, where $S_{J^i_{p}}$ is as given in (\ref{def:S_J_Hn}). Now
depending on connectedness between $J^k_{p}$'s, the following three cases arise in (\ref{eqn:w_p^4}):
\vskip5pt
\noindent \textbf{Case I.} \textbf{At least one of $J^k_{p}$ for $k=1,2,3,4$, is not connected with the remaining ones:} 

 Since one $J^k_{p}$ is not connected with the others therefore due to the independence of  entries, we get $\E[w_{p}(t) -w_{p}(s)]^4=0$.
 \vskip5pt
 \noindent \textbf{Case II.} \textbf{$J^1_{p}$ is connected with only one of $J^2_{p}, J^3_{p}, J^4_{p}$ and the remaining two of $J^2_{p}, J^3_{p}, J^4_{p}$  are connected only with each other:} Without loss of generality, we assume $J^1_{p}$ is connected only with $J^2_{p}$ and $J^3_{p}$ is connected only with $J^4_{p}$.  
Under this situation the terms in the right hand side of (\ref{eqn:w_p^4}) can be written as
\begin{align} \label{eqn:case II} 
%\E[w_{p}(t) -w_{p}(s)]^4 & = 
&\frac{(t-s)^2}{n^4 {b_n}^{2p-2}} \sum_{i_1, i_2, i_3, i_4=1}^n  \sum_{ A_{p}, A_{p}} \E \big[ (Z_{J^1_{p}} - \E Z_{J^1_{p}}) (Z_{J^2_{p}} - \E Z_{J^2_{p}}) \big] I(i_1, J^1_p) I(i_2, J^2_p) \nonumber \\ 
& \qquad \qquad \qquad \qquad \qquad \sum_{A_{p}, A_{p}} \E \big[ (Z_{J^3_{p}} - \E Z_{J^3_{p}}) (Z_{J^4_{p}} - \E Z_{J^4_{p}}) \big]I(i_3, J^3_p) I(i_4, J^4_p) = \theta, \mbox{ say}.
\end{align}
Now from the definition of $Z_{J_{p}}$, given in (\ref{eqn:Z_p}) and $\E(x_i)= \E(y_i) = 0$, we get

\begin{align} \label{eqn:B_p_2}
& \sum_{ A_{p}, A_{p}} \big|\E \big[ (Z_{J^1_{p}} - \E Z_{J^1_{p}}) (Z_{J^2_{p}} - \E Z_{J^2_{p}}) \big] I(i_1, J^1_p) I(i_2, J^2_p) \big| \nonumber \\
& \qquad = \sum_{( J^1_{p}, J^2_{p}) \in B_{P_2}} \big|\E \big[ (Z_{J^1_{p}} - \E Z_{J^1_{p}}) (Z_{J^2_{p}} - \E Z_{J^2_{p}}) \big] I(i_1, J^1_p) I(i_2, J^2_p) \big| 
\leq \sum_{( J^1_{p}, J^2_{p}) \in B_{P_2}} \alpha,
\end{align}
where $\alpha$ is a positive constant and last inequality arises as $x_i \distas{D} N(0, t-s)$, $y_i \distas{D} N(0, s)$, and $B_{P_2}$ is defined as 
\begin{align*}B_{P_2}=\{( J^1_{p}, J^2_{p}) \in  A_{p} \times A_{p} &: S_{J^1_{p}} \cap S_{J^2_{p}} \neq \emptyset \mbox{ and each entries of } S_{J^1_{p}} \cup S_{J^2_{p}}\\
&\quad \mbox{ has multiplicity greater than or equal to two}\},
\end{align*}
with $S_{J^i_p}$ is as given in (\ref{def:S_J_Hn}). It is easy to observe that $|B_{P_2}| = O({b_n}^{p-1})$.

Now from (\ref{eqn:case II}) and (\ref{eqn:B_p_2}), we get
 
\begin{align} \label{case II,(t-s)}
%\E[w_{p}(t) -w_{p}(s)]^4 
|\theta|\leq \frac{(t-s)^2}{ n^4 {b_n}^{2p-2}} \sum_{i_1, i_2, i_3, i_4=1}^n  \alpha^2 |B_{P_2}|^2 = (t-s)^2 \alpha^2 O(1) \leq M_1 (t-s)^2,
\end{align} 
%The above last equality arises as $|B_{P_2}| = O({b_n}^{p-1})$.
where $M_1$ is a positive constant depending only on $p$ and $T$.\\
 \noindent \textbf{Case III.} \textbf{$J^1_{p}, J^2_{p}, J^3_{p}, J^4_{p}$ are connected:}  
Since $\E(x_{i})=0$ and $\E(y_{i})=0$, therefore
\begin{align*} 
\sum_{ A_{p}, A_{p}, A_{p}, A_{p}} & \E \big[ \prod_{k=1}^4 (Z_{J^k_{p}} - \E Z_{J^k_{p}}) \big] 
 = \sum_{( J^1_{p}, J^2_{p}, J^3_{p}, J^4_{p}) \in B_{P_4}} \E \big[ \prod_{k=1}^4 (Z_{J^k_{p}} - \E Z_{J^k_{p}}) \big],
\end{align*}
 where $B_{P_4}$ is defined as,
 \begin{align*} 
 B_{P_4}=\{( J^1_{p}, J^2_{p}, J^3_{p}, J^4_{p}) \in & A_{p} \times A_{p} \times A_{p} \times A_{p}  : S_{J^1_{p}} \cap S_{J^2_{p}} \cap S_{J^3_{p}} \cap S_{J^4_{p}}  \neq \emptyset  \mbox{ and each entries of }  \\
& \qquad  S_{J^1_{p}} \cup S_{J^2_{p}} \cup S_{J^3_{p}} \cup S_{J^4_{p}}  \mbox{ has multiplicity greater than or equal to two}
\}.
\end{align*}

 Observe that $|B_{P_4}| = o({b_n}^{2p-2})$. Hence
   
 \begin{align} \label{eqn:case III,Z}
\frac{(t-s)^2}{n^4 {b_n}^{2p-2}} \sum_{i_1, i_2, i_3, i_4=1}^n  \sum_{ A_{p}, A_{p}, A_{p}, A_{p}} \Big|\E\big[ \prod_{k=1}^4 (Z_{J^k_{p}} - \E Z_{J^k_{p}})\big] I(i_k, J^k_p) \Big| &\leq \frac{(t-s)^2}{n^4 {b_n}^{2p-2}} \sum_{i_1, i_2, i_3, i_4=1}^n \beta o({b_n}^{2p-2})\nonumber\\
& = \beta (t-s)^2  o(1)  \nonumber\\
& \leq  M_2 (t-s)^2,
 \end{align}
where  $M_2$ is a positive constant depending only on $p, T$. 

Now combining \eqref{case II,(t-s)} and (\ref{eqn:case III,Z}), it follows that there exists a positive constant $M_T$, depending only on $p, T$ such that
\begin{align} \label{eqn:M_T}
%\E[w_p(t) - w_p(s)]^4 & \leq M^T_7 (t-s)^2 + M^T_{10} (t-s)^2 \nonumber \\
\E[w_{p}(t) -w_{p}(s)]^4 & \leq M_T (t-s)^2 \ \ \ \forall \ n \in \mathbb{N} \ \mbox{and } t,s \in[0,\ T].
\end{align}
This completes the proof of Proposition \ref{pro:tight} with $\alpha = 4$ and $\beta = 1$.
\end{proof}
\begin{remark} For any fixed $N \in \mathbb{N}$, (\ref{eqn:M_T}) implies
	%we can conclude that 
	\begin{equation*}
	\{ w_{p}(t) ; t \geq 0, 1 \leq p \leq N\} \stackrel{\mathcal D}{\rightarrow} \{W_{p}(t) ; t \geq 0 , 1 \leq p \leq N\} \quad \mbox{ as } n \tends \infty. 
	\end{equation*}  
	But the process convergence of $\{w_{p}(t) ; t \geq 0, p \geq 2\}$ does not follow from the proof of Theorem \ref{thm:hankprocess}, because the constant $M_T$ (depends on $p$) of (\ref{eqn:M_T}) may tends to infinity as $p \tends \infty$.
\end{remark}
The following remark provides an application of Theorem \ref{thm:hankprocess} in the study of the limiting law of $\sup_{0 \leq t \leq T}|w_p(t)|$: 
\begin{remark} \label{rem:sup_wp(t)_convergence_Hp}
	Suppose $ p\geq 2 $ is an even positive integer and $T$ is a positive real number. Then as $n\to\infty$
	\begin{equation} \label{eqn:sup_wp(t)_convergence}
	 \sup_{0 \leq t \leq T} |w_{p}(t)| \stackrel{\mathcal D}{\rightarrow}  \sup_{0 \leq t \leq T} |W_{p}(t)|,
	\end{equation}
	where $\{W_{p}(t);  t \geq 0\}$  is as defined  in Theorem \ref{thm:hankprocess}.
	The convergence of (\ref{eqn:sup_wp(t)_convergence}) follows from the continuous mapping theorem and Theorem \ref{thm:hankprocess}.
\end{remark}

\begin{remark} \label{rem:poly_test_function_Hp}
Theorem \ref{thm:hankprocess} and Remark \ref{rem:sup_wp(t)_convergence_Hp} can also be generalised for even degree polynomial test function. The proofs will be similar to the proof of Theorem \ref{thm:hankprocess} and Remark \ref{rem:sup_wp(t)_convergence_Hp}, respectively.
\end{remark}

In the following remark, we discuss the process convergence when the entry $\{a_0(t);t\ge 0\}$ is the standard Brownian motion.

\begin{remark} \label{rem:hankprocess_a_0_proof}
Let $\{a_0(t); t\geq 0\}$ be a  standard Brownian motion  and independent of $\{a_n(t);t\ge 0\}_{n\ge 1}$. For $p\geq 1$, define 
\begin{align*}\label{eqn:w_p(t)_tilde}
\tilde{A}_n(t) & := A_n(t) + \frac{a_0(t)}{\sqrt{b_n}} I_n \ \mbox{ and } \
   \tilde{w}_{p}(t)  := \frac{\sqrt{b_n}}{n}\big(\Tr(\tilde{A}_n(t))^p-\E\Tr(\tilde{A}_n(t))^p\big),
\end{align*}
where $A_n(t)$ is the Hankel matrix as defined in (\ref{def:A_n(t)}) and $I_n$ is an $n \times n$ identity matrix.

 Suppose $ p\geq 2$ is a positive even integer. Then as $n\to\infty$
	\begin{equation} \label{eqn:hankprocess_a_0}
	\{ \tilde{w}_{p}(t) ; t \geq 0\} \stackrel{\mathcal D}{\rightarrow} \{\tilde{W}_{p}(t) ; t \geq 0\}, 
	\end{equation}
	where $\{\tilde{W}_{p}(t);  t \geq 0\}$  is a Gaussian process and 
	\begin{equation*} \label{def:tilde_W_p(t)}
 \tilde{W}_p(t) = W_p(t) + p t^{\frac{p-1}{2}} R_{p-1} a_0(t),
 \end{equation*}
with $\{W_{p}(t);  t \geq 0\}$  as in Theorem \ref{thm:hankprocess} and  $R_k = \lim\limits_{n\to\infty} \E \big[\frac{1}{n}\Tr \big( A_n(1)\big)^{k} \big]$. 

Note that, to prove (\ref{eqn:hankprocess_a_0}),  it is enough to prove Proposition \ref{pro:finite convergence} and Proposition \ref{pro:tight} for the process $\{\tilde{w}_p(t);t\ge 0\}$. The idea of proof is same as the proof of Theorem \ref{thm:hankprocess}. One can also see, \cite[Theorem 4]{m&s_toeplitz_2020},  where Maurya and Saha have derived similar type of result for band Toeplitz matrices. 
\end{remark}

In the following remark, we discuss the process convergence of $\{w_{p}(t);t \geq 0\}$ for odd $ p\geq 1$.

\begin{remark} \label{rem:hankprocess_wp_odd_proof}
 By the similar ideas as used in the proof of Theorem \ref{thm:hank_odd_wp}, one can derive the following result for any odd integer $ p\geq 1$ and $t \geq 0$,
	\begin{equation*}
	 w_{p}(t)  \stackrel{d}{\rightarrow} 0 \mbox{ as } n \to \infty,
	\end{equation*}
	which implies the finite dimensional convergence of $\{w_{p}(t);t \geq 0\}$ for odd $ p\geq 1$. The proof of tightness of $\{w_{p}(t);t \geq 0\}$ for odd $ p\geq 1$ goes similar to the proof of tightness of $\{w_{p}(t);t \geq 0\}$ for even $ p\geq 2$ (Proof of Proposition \ref{pro:tight}) and hence we conclude that for odd $p$,
		\begin{equation*}
	\{ w_{p}(t) ; t \geq 0\} \stackrel{\mathcal D}{\rightarrow} 0  \mbox{ as } n \to \infty.
	\end{equation*}
\end{remark}

\section{Proof of Theorem \ref{thm:hank_odd_wp}}\label{sec:thm:hank_odd_wp}

We  prove Theorem \ref{thm:hank_odd_wp} by showing that $\var(w_p) \rightarrow 0$ as $n \rightarrow \infty$.
\begin{proof}[Proof of Theorem \ref{thm:hank_odd_wp}: ]
First note from (\ref{eqn:w_p}) that for $p=1$, $w_1 =0$ and the result is trivial.

Using the trace formula from Result \ref{res:trace_Hn}, we get for every odd $p \geq 3$,
\begin{equation} \label{eqn:wp_trace}
w_p = \frac{\sqrt{b_n}}{n} \bigl\{ \Tr(A_n)^{p} - \E[\Tr(A_n)^{p}]\bigr\} = \frac{1}{n} b_{n}^{-\frac{p-1}{2}} \sum_{i=1}^{n} \sum_{J \in A_{p,i}} I_{J}\left(x_{J}-\E\left[x_{J}\right]\right), 
\end{equation}
where $I_J = \prod_{\ell=1}^p \chi_{[1, n]}\left(i-\sum_{q=1}^{\ell} j_{q}\right)$, $x_J= \prod_{\ell=1}^{p} x_{j_\ell}$ and for $1 \leq i \leq n$, define
    \begin{equation*}
    A_{p, i}=\left\{\left(j_{1}, \ldots, j_{p}\right) \in\left\{\pm 1, \ldots, \pm b_{n}\right\}^{p}: \sum_{k=1}^{p} (-1)^{k} j_{k}=2i-1-n\right\}.
\end{equation*}
	Since $\E(w_p)=0$ therefore $\Cov(w_p, w_q)= \E(w_p w_q)$ and hence from (\ref{eqn:wp_trace}), we get

	\begin{align}\label{eqn:cov}
	\Cov(w_p, w_q)=\E[w_p w_q]
%	&= \frac{b_n}{n^2} \Big\{ \E[\Tr(A_n)^{p}\Tr(A_n)^{q} ]- \E[\Tr(A_n)^{p}]\E[\Tr(A_n)^{q}]   \Big\} \nonumber\\
	& = \frac{1}{n^{2}} b_{n}^{-\frac{p+q}{2}+1} \sum_{i_1, i_2=1}^n \sum_{J_1, J_2} \big\{ \E\left[x_{J_1} x_{J_2}\right]- \E[x_{J_1}]\E[x_{J_2}] \big\} I_{J_1} I_{J_2},
	\end{align}
where $J_1 \in A_{p,i_1}$ and $ J_2 \in A_{q,i_2}$. 

Notice that the summand in (\ref{eqn:cov}) will be non-zero only if each element in $S_{J_1} \cup S_{J_2}$ have multiplicity at least two and $S_{J_1} \cap S_{J_2} \neq \emptyset$, where for a vector $J=\left(j_{1}, j_{2}, \ldots, j_{p}\right) \in A_{p, i},$ 
the multi-set $S_{J}$ is defined as
\begin{equation*} 
  S_{J}=\left\{ j_{1},j_{2}, \ldots,j_{p} \right\}. 
\end{equation*} 
Since $p$ and $q$ are odd, there exist at least one element each in $S_{J_1}$ and $S_{J_2}$ with odd multiplicity in $S_{J_i}$. Notice that the number of terms with at least one component having odd multiplicity greater than 1 in $S_{J_i}$,  is at most $O(b_n^{\frac{p+q-3}{2}})$, leading to a zero contribution in (\ref{eqn:cov}) as $n \rightarrow \infty$. A similar argument also shows that terms with at least one $j_k$ having multiplicity greater than 2 in $S_{J_1}\cup S_{J_2}$ also lead to zero contribution in (\ref{eqn:cov}) as $n \rightarrow \infty$. Thus we restrict ourselves to the case where the multiplicity of every element in $S_{J_1} \cup S_{J_2}$ is 2. Since there exist elements of multiplicity $1$ in both $S_{J_1}$ and $S_{J_2}$, independence of the input sequence implies that $\E[x_{J_1}]=0=\E [x_{J_2}]$. Therefore from (\ref{eqn:cov}),
\begin{align}\label{eqn:covlimit}
     \lim_{n \rightarrow \infty} \Cov(w_p, w_q)  
   %  \lim_{n \rightarrow \infty}\E\left[w_{p} w_{q}\right] \nonumber\\
     = \displaystyle \lim_{n \rightarrow \infty} \frac{1}{n^{2}} b_{n}^{-\frac{p+q}{2}+1} \sum_{i_1, i_2=1}^n \sum_{\pi \in \PP_2(p+q)} \sum_{J_\pi}  \E\left[x_{J_1} x_{J_2}\right] I_{J_1} I_{J_2} ,
\end{align}
where $\PP_2(p+q)$ is the set of all pair partitions of $\{ 1,2,\ldots , p+q \}$ and $J_{\pi}$ is the set of all vectors $(J_1,J_2)$ corresponding to $\pi$ with $J_1 \in A_{p,i_1}$ and $J_2 \in A_{q,i_2}$. 

The rest of the proof is done by considering the following cases:

\noindent \textbf{Case I.} \textbf{$\mathbf{i_1 + i_2 = n+1}$:}
Consider $J_1 \in A_{p,i_1}, J_2 \in A_{q,i_2}$. Since $p$ is odd, there exist a component in $J_1$ with multiplicity 1 in $S_{J_1}$, whose value would be determined by the other $ p-1$ terms. Hence the number of terms $(J_1,J_2) \in A_{p,i_1} \times A_{q,i_2}$ such that $\E[x_{J_1}x_{J_2}]$ is non-zero is of order $O\left(b_{n}^{\frac{p+q-2}{2}}\right)$. Hence, the sum of such terms in (\ref{eqn:covlimit}) is of order $O\left(\frac{1}{n^2} \times n \times \frac{1}{b_n^{\frac{p+q-2}{2}}} \times b_n^{\frac{p+q-2}{2}}\right)=O\left(\frac{1}{n}\right)$. 

\noindent \textbf{Case II.} \textbf{$\mathbf{i_1 + i_2 \neq n+1}$:} 
For a pair-partition $\pi \in \PP_{2}(p+q)$, consider an associated $J$ such that $J=(j_1,j_2, \ldots , j_{p+q}) = (J_1,J_2)$ where $J_1 = (j_1, j_2, \ldots, j_p)$ and $J_2 = (j_{p+1},j_{p+2}, \ldots , j_{p+q}); J_1 \in A_{p,i_1}$, $J_2 \in A_{q,i_2}$ and $\E[x_{J_1}x_{J_2}]$ is non-zero. 

\textbf{subcase (i). $i_1 \neq i_2$:}
 
Without loss of generality assume that $2i_1-n-1 \neq 0$. Since $\sum_{i=1}^p (-1)^k j_k = 2i_1-n-1$, $\sum_{i=1}^{p+q}(-1)^k j_k = 2(i_1+i_2)-2n-2$ and both of the right hand sides are non-zero, there is a loss of two degree of freedom. Thus the sum of such terms in (\ref{eqn:covlimit}) is $O\left(\frac{1}{n^2} \times n^2 \times \frac{1}{b_n^{\frac{p+q-2}{2}}} \times b_n^{\frac{p+q-4}{2}}\right)=O\left(\frac{1}{b_n}\right)$.

\textbf{subcase (ii). $i_1 = i_2$:}

By a similar argument, as given in Case I, we get that the order of the summation of such terms in  (\ref{eqn:covlimit}) would be $O\left(\frac{1}{n^{2}} \times n \times \frac{1}{b_n^{\frac{p+q-2}{2}}} b_{n}^{\frac{p+q-2}{2}}\right)=O\left(\frac{1}{n}\right)$. 

Hence on combining Case I and Case II, we get that for all odd integers $p, q \geq 2$
\begin{align*} \label{eqn:lim_cov_0}
\lim_{n \rightarrow \infty} \Cov(w_p, w_q) = 0, 
\end{align*}
which shows $\var(w_p)=0$ and hence $w_p \stackrel{d}{\longrightarrow}  0$. This completes the proof of Theorem \ref{thm:hank_odd_wp}.
\end{proof}
\begin{remark} \label{rem:poly_test_function_Hn}
Theorem \ref{thm:hank_odd_wp} can also be generalised for odd degree polynomial test functions. 
\end{remark}

\noindent{\bf Acknowledgement:} We express our most sincere thanks to Prof. Koushik Saha for reading the initial and final draft and providing helpful advice.

\providecommand{\bysame}{\leavevmode\hbox to3em{\hrulefill}\thinspace}
\providecommand{\MR}{\relax\ifhmode\unskip\space\fi MR }
\providecommand{\MRhref}[2]{%
  \href{http://www.ams.org/mathscinet-getitem?mr=#1}{#2}
}
\providecommand{\href}[2]{#2}

\end{document}